\documentclass[11pt]{article}
\usepackage{amssymb,amsmath, amsthm}
\usepackage{bbm}

\usepackage[hidelinks]{hyperref} 
\usepackage{makeidx}
\usepackage{enumerate}
\usepackage{stmaryrd}

\usepackage{dsfont}

\title{Character rigidity for lattices and commensurators}
\author{Darren Creutz and Jesse Peterson}

\date{}

\newtheorem{thm}{Theorem}[section]
\newtheorem{thma}{Theorem}

\newtheorem{cora}[thma]{Corollary}
\newtheorem{prop}[thm]{Proposition}
\newtheorem{cor}[thm]{Corollary}
\newtheorem{lem}[thm]{Lemma}

\newcommand{\actson}{{\curvearrowright}}

\newcommand{\rpf}[2]{#1\hspace{-.175em}\sslash\hspace{-.2em}#2}

\newcommand{\normal}{\hspace{0.05em}\triangleleft\hspace{0.05em}}

\begin{document}

\maketitle

\begin{abstract}
We prove an operator algebraic superrigidity statement for homomorphisms of irreducible lattices, and also their commensurators, from certain higher-rank groups into unitary groups of finite factors. This extends the authors' previous work regarding non-free measure-preserving actions, and also answers a question of Connes for such groups.
\end{abstract}

\section{Introduction}

A seminal result in the theory of semisimple groups and their lattices is Margulis' superrigidity theorem \cite{margulissuperrigidity}: Let $\Gamma$ be an irreducible lattice in a center free higher-rank semisimple group $G$ with no compact factors, let $H$ be a simple algebraic group over a local field and let $\pi: \Gamma \to H$ be a homomorphism whose image is Zariski dense in $H$, then either $\pi(\Gamma)$ is precompact or else $\pi$ extends to a continuous homomorphism $G \to H$.

Motivated by a conjecture of Selberg \cite{selberg}, Margulis developed the superrigidity theorem as the central ingredient in the proof of the Arithmeticity Theorem \cite{margulisarith}, which states that every irreducible lattice in a higher-rank semisimple Lie or algebraic group is, in a suitable sense, the  integer points of an algebraic group over a global field.  Since then, the phenomenon of superrigidity has found a wide array of applications, notably Zimmer's orbit equivalence rigidity stating that if two such lattices admit probability-preserving actions that are orbit equivalent then the ambient groups are locally isomorphic (a consequence of the cocycle superrigidity theorem \cite{zimmersuperrigidity} generalizing Margulis' work), and Furman's measure equivalence theorem \cite{furmanrigidity} stating that if a countable group is measure equivalent to a lattice in a higher-rank simple group then that group is in fact itself also a lattice.

There is a rich analogy between the interaction of a lattice in a group $\Gamma < G$, and the interaction between a countable group in its von Neumann algebra $\Gamma < \mathcal U(L\Gamma)$. In both situations the ``analytic'' properties of $G$ or $L\Gamma$ are often reflected in corresponding properties of $\Gamma$. For example, Connes exhibited the first rigidity phenomenon in II$_1$ factors, the existence of factors with countable fundamental group \cite{connescountable}, by exploiting Kazhdan's property (T) \cite{kazhdan} for $\Gamma$. This analogy was made more precise when Connes introduced his theory of correspondences \cite{connescorrespondences, popacorrespondences}, which provided the proper substitute for the representation theory of a finite von Neumann algebra. Using this theory one is able to define analytic properties of a finite von Neumann algebra such as amenability, property (T), the Haagerup property, etc. More importantly, by inducing or restricting one is able to relate representations of the group $\Gamma$ with correspondences for $L\Gamma$ and in this way show that such properties for the von Neumann algebra $L\Gamma$ are shared by $\Gamma$ (e.g., \cite{connesclassification, connessurvey, connesjonesT, choda}).

Based on the strong rigidity result of Mostow \cite{mostowstrongrigidity}, the superrigidity result of Margulis, and the cocycle version due to Zimmer, Connes suggested that the analogy could be pursued further and that there should be a similar superrigidity phenomenon for such groups embedded in their group von Neumann algebras. Connes further suggested that the first difficulty is to understand the role of the Poisson boundary in the setting of operator algebras. (See the discussion on page 86 in \cite{jonesproblems}).

The first examples of lattices in higher-rank groups where this ``operator algebraic superrigidity'' was verified were obtained by Bekka \cite{bekka} who showed that this holds for the groups $SL_n(\mathbb Z)$, for $n \geq 3$. Further examples were found in \cite{petersonthom} where the same results were obtained for the groups $SL_2(A)$ where $A = \mathcal O$ is a ring of integers (or, more generally, $A = \mathcal O S^{-1}$ a localization) with infinitely many units. Despite this initial progress, the proofs in \cite{bekka} and \cite{petersonthom} rely heavily on the structure of $SL_n$, e.g., by explicitly working with unipotent elements and using the congruence subgroup property, and as such do not appear to generalize to arbitrary irreducible lattices. The purpose of this paper is to provide the first examples of higher-rank groups $G$ such that operator algebraic superrigidity holds for arbitrary irreducible\footnote{We consider a lattice $\Gamma < G$ to be irreducible if for each non-compact closed normal subgroup $G_0 \lhd G$, we have that $\Gamma$ projects densely onto the quotient $G/G_0$. In the context of lattices in products of groups this is sometimes referred to as strong irreducibility.} lattices.

\begin{thma}[Operator Algebraic Superrigidity for lattices]\label{thm:opalgsuperrigidty}
Suppose $G$ is a semisimple connected Lie group with trivial center and no compact factors, such that at least one factor is higher-rank, and suppose $H$ is a non-compact totally disconnected semisimple algebraic group over a local field with trivial center and no compact factors. Let $\Gamma < G \times H$ be an irreducible lattice, and suppose $\pi: \Gamma \to \mathcal U(M)$ is a representation into the unitary group of a finite factor $M$ such that $\pi(\Gamma)'' = M$. Then either $M$ is finite dimensional, or else $\pi$ extends to an isomorphism $L\Gamma \xrightarrow{\sim} M$.
\end{thma}

An example of $G$ and $H$ where the hypotheses in the previous theorem are satisfied is $G = PSL_n(\mathbb R)$ and $H = PSL_n(\mathbb Q_p)$, for $n \geq 3$ and $p$ a prime. The theorem above actually holds in a greater generality, e.g., in many cases $G$ itself can also be totally disconnected (see Section~\ref{sec:mainresults} for the full generality). The above superrigidity result is a consequence of a corresponding superrigidity result for commensurators. To state this result we first recall that if $\Gamma < \Lambda$ is an inclusion of countable groups, then we say that $\Lambda$ commensurates $\Gamma$ if $[ \Gamma: \Gamma \cap \lambda \Gamma \lambda^{-1} ] < \infty$, for each $\lambda \in \Lambda$. To such an inclusion we may consider the homomorphism of $\Lambda$ into ${\rm Symm}(\Lambda/ \Gamma)$ given by left multiplication. The relative profinite completion of $\Lambda$ with respect to $\Gamma$ is the closure of the image of $\Lambda$ in ${\rm Symm}(\Lambda/ \Gamma)$ where the latter is endowed with the topology of pointwise convergence (see, e.g., \cite{schlichting, shalomwillis}). The relative profinite completion, denoted $\rpf{\Lambda}{\Gamma}$, is a totally disconnected group that is locally compact, with the closure of the image of $\Gamma$ in ${\rm Symm}(\Lambda/ \Gamma)$ being a compact open subgroup.

\begin{thma}[Operator Algebraic Superrigidity for commensurators]\label{thm:oascommensurator}
Let $G$ be as in Theorem~\ref{thm:opalgsuperrigidty} and suppose $\Lambda < G$ is a countable dense subgroup that contains and commensurates a lattice $\Gamma < G$ such that $\rpf{\Lambda}{\Gamma}$ is a product of simple groups with the Howe-Moore property.

If $\pi: \Lambda \to \mathcal U(M)$ is a finite factor representation such that $\pi(\Lambda)'' = M$, then either $M$ is finite dimensional, or else $\pi$ extends to an isomorphism $L \Lambda \xrightarrow{\sim} M$.
\end{thma}

These results, describing the types of homomorphisms from a lattice or commensurator into the unitary group of a finite factor, should be contrasted with Popa's rigidity results \cite{popastrongrigidity1, popastrongrigidity2, popasuperrigidityproducts, popasuperrigidityT} where he shows, for example, that under certain ``malleability'' conditions for a measure-preserving action, a cocycle for this action into the unitary group of a finite von Neumann algebra is always cohomologous to a homomorphism. Thus, combining our result with Popa's superrigidity results (see also \cite{petersonsinclair}) one obtains a further rigidity statement about cocycles for such actions. For example, we obtain the following result.

\begin{thma}
Let $\Lambda$ and $G$ be as in Theorem~\ref{thm:oascommensurator}, and consider the Bernoulli shift action $\Lambda \actson [0, 1]^\Lambda$. If $M$ is a separable finite von Neumann algebra, and $\alpha: \Lambda \times [0, 1]^\Lambda \to \mathcal U(M)$ is a cocycle, then there exists a von Neumann subalgebra $N \subset M$, a central projection $p \in \mathcal Z(N)$, an isomorphism $\theta_1: L\Lambda \to Np$ (if $p \not= 0$), and a homomorphism $\theta_2: \Lambda \to \mathcal U(N p^\perp)$ such that $\theta_2(\Lambda)$ is precompact, and such that $\alpha$ is cohomologous to the homomorphism $\lambda \mapsto \theta_1(u_\lambda)p + \theta_2(\lambda)p^\perp$.
\end{thma}

As noted in \cite{bekka}, operator algebraic superrigidity can also be described in the framework of characters. A character on a discrete group $\Lambda$ is a conjugation invariant function $\tau$ that is of positive type, and is normalized so that $\tau(e) = 1$ (or equivalently, a tracial state on the full group $C^*$-algebra $C^*\Lambda$). If $M$ is a finite von Neumann algebra with trace $\tau$, and $\pi: \Lambda \to \mathcal U(M)$ is a representation then $\lambda \mapsto \tau(\pi(\lambda))$ gives a character. Moreover, the GNS-construction shows that every character arises in this way. It is also not hard to see that $M$ may be chosen to be completely atomic ($M$ is isomorphic to a direct sum of matrix algebras) if and only if $\tau$ is almost periodic in $\ell^\infty \Lambda$ ( the set of translations $\{ x \mapsto \tau(\lambda x) \mid \lambda \in \Lambda \}$ is precompact in $\ell^\infty\Lambda$). The space of characters forms a Choquet simplex, and the extreme points correspond to representations that generate a finite factor \cite{thoma1}.

Theorem~\ref{thm:oascommensurator} is therefore equivalent to the following rigidity result for characters, which reduces the classification to that of continuous characters on the Bohr compactification. There is a similar statement for Theorem~\ref{thm:opalgsuperrigidty}.

\begin{thma}[Character Rigidity for commensurators]
Let $\Lambda$ and $G$ be as in Theorem~\ref{thm:oascommensurator}. If $\tau: \Lambda \to \mathbb C$ is an extreme point in the space of characters then either $\tau$ is almost periodic, or else $\tau = \delta_e$.
\end{thma}

For finite or compact groups the study of characters has a long and successful history, dating back to the works of Frobenius, Schur, Peter-Weyl, and others. Classification of characters for non-compact groups goes back to Segal and von Neumann's result in \cite{segalvonneumann} (see also \cite{kadisonsingerreps}) where they show that connected simple Lie groups have no non-trivial continuous homomorphisms into a finite factor, and hence have no non-trivial continuous characters. For countably infinite groups the study of characters was initiated by Thoma \cite{thoma1, thoma2, thoma3} who gave a classification of extreme characters for the group of finite permutations of $\mathbb N$. Since then, classification results for characters on non-compact groups have been extended to a wide range of settings. The emphasis first being on more ``classical'' type groups, e.g., \cite{kirillov, ovcinikov, skudlarek, voiculescucharacters, vershikkerov1, vershikkerov2, boyer, boyer2, boyer3}, and then more recently to the less ``classical'' setting, e.g., \cite{dudkonessonov1, dudkonessonov2, dudkonessonov3, dudkofullgroup, dudkomedynetslimits, dudkomedynetsthompson, enomotoizumi}. The only previous classification results focusing on lattices were obtained in \cite{bekka} and \cite{petersonthom}. In the time since this paper was written the operator algebraic superrigidity statement in Theorem~\ref{thm:opalgsuperrigidty} has been shown by the second author to also hold in the case when $\Gamma$ is a lattice in a general higher-rank semi-simple Lie group, so long as the ambient group has property (T), e.g., if the group is simple. The case when the lattices is in a product of rank 1 groups with the Haagerup property, e.g., lattices in $SL_2(\mathbb R) \times SL_2(\mathbb R)$, remains open. 

Another consequence of operator algebraic superrigidity is that it implies a rigidity phenomenon for the stabilizers of measure-preserving actions. Specifically, given a measure-preserving action of the lattice $\Gamma$, one naturally obtains a homomorphism of the group into the von Neumann algebra associated to the orbit equivalence relation \cite{feldmanmooreII} (see also \cite{vershik}). Applying operator algebraic superrigidity to this setting we obtain the main result from \cite{creutzpeterson}.

\begin{cora}[\cite{creutzpeterson, creutz}]\label{cora:nonfree}
Let $\Lambda$ and $G$ be as in Theorem~\ref{thm:oascommensurator}. Then any probability measure-preserving ergodic action of $\Lambda$ on a standard Lebesgue space is essentially free.
\end{cora}

The previous corollary (which generalizes the normal subgroup theorem in \cite{creutzshalom}), or rather its proof from \cite{creutzpeterson}, was the starting point for this current work.

\section{Preliminaries and motivation}\label{sec:prelim}

In this section we will recall some of the notions from ergodic theory and von Neumann algebras that we will use in the sequel, and we will also outline our argument. For a more detailed review of the ergodic theory of semisimple groups and von Neumann algebras we refer the reader to \cite{zimmerbook} and \cite{dixmier}.

Let $N$ be a finite von Neumann algebra with a normal faithful trace $\tau$.  The trace on $N$ provides us with a positive definite inner product on $N$ given by $\langle x, y \rangle = \tau(y^*x)$, and we denote by $L^2(N, \tau)$ the Hilbert space completion of $N$ with respect to this inner product. We also have a norm on $N$ given by $\| x \|_1 = \tau(|x|)$, and we denote by $L^1(N, \tau)$ the Banach space completion of $N$ with this norm. Note that $\| x \|_1 \leq \| x \|_2 \leq \| x \|$ and so we consider inclusions $N \subset L^2(N, \tau) \subset L^1(N, \tau)$, moreover, each of these spaces is an $N$-bimodule, the trace extends continuously to $L^1(N, \tau)$, and we may identify $L^1(N, \tau)$ with the predual of $N$, by considering the linear functionals $N \ni x \mapsto \tau( x \xi)$, for $\xi \in L^1(N, \tau)$.

Left multiplication gives us a faithful representation of $N$ in the space of bounded operators on $L^2(N, \tau)$, and so we will always consider $N \subset \mathcal B(L^2(N, \tau))$. Since we are also considering $N \subset L^2(N, \tau)$, to avoid confusion we shall sometimes use the notation $\hat x \in L^2(N, \tau)$ to identify an operator $x \in N$ as an element in $L^2(N, \tau)$. 

Since $\tau$ is a trace it follows that conjugation $x \mapsto x^*$ on $N$ extends to an anti-linear isometry $J: L^2(N, \tau) \to L^2(N, \tau)$. Note that if $x, y \in N$ then we have $J x^* J \hat y = \widehat{yx}$ so that $J x^* J$ is multiplication on the right by $x$. Thus, $J N J \subset N'$ and in fact we have $J N J = N'$.

If $\Gamma$ is a countable group and we consider the left-regular representation $\lambda: \Gamma \to \mathcal U(\ell^2\Gamma)$, then the group von Neumann algebra $L\Gamma$ is the von Neumann algebra generated by this representation, i.e., $L\Gamma = \{ \lambda(\Gamma) \}''$ \cite{murrayvonNeumann4}. Since $\Gamma$ is discrete, the group von Neumann algebra will be a finite von Neumann algebra with a canonical normal faithful trace given by $\tau(x) = \langle x \delta_e, \delta_e \rangle$. As $\delta_e$ is a tracial and cyclic vector we may identify $\ell^2\Gamma$ with $L^2(L\Gamma, \tau)$ in such a way that the vector $\delta_\gamma$ corresponds to the vector $\widehat{\lambda(\gamma)}$, for each $\gamma \in \Gamma$. Under this identification the von Neumann algebra $L\Gamma$, as a subspace of $\ell^2\Gamma$, can then be described as the space of convolvers $\{ \xi \in \ell^2 \Gamma \mid \xi * \eta \in \ell^2 \Gamma, {\rm \ for \ all \ } \eta \in \ell^2 \Gamma \}$.

If we are given a quasi-invariant action on a probability space $\Gamma \actson (B, \eta)$ then we also have an action of $\Gamma$ on the space of measurable functions given by $\sigma_\gamma(a) = a \circ \gamma^{-1}$. The Koopman representation of this action on $L^2(B, \eta)$ is the unitary representation given by the formula $\sigma_\gamma^0(\xi) = \sigma_\gamma(\xi) \sqrt{ \frac{d \gamma \eta}{d \eta}}$. Note that by considering point-wise multiplication we may realize $L^\infty(B, \eta)$ as a von Neumann subalgebra of $\mathcal B(L^2(B, \eta))$, and under this identification the unitaries $\sigma_\gamma^0$ normalize $L^\infty(B, \eta)$. Specifically, we have the formula $\sigma_\gamma^0 a \sigma_{\gamma^{-1}}^0 = \sigma_\gamma(a)$, for all $a \in L^\infty(B, \eta) \subset \mathcal B(L^2(B, \eta))$, and $\gamma \in \Gamma$.  The group-measure space construction is the von Neumann algebra $L^\infty(B, \eta) \rtimes \Gamma \subset \mathcal B(L^2(X, \eta) \overline \otimes \ell^2\Gamma)$ generated by $L^\infty(B, \eta)$, together with the unitary operators $\{ \sigma_\gamma^0 \otimes \lambda(\gamma) \mid \gamma \in \Gamma \}$ \cite{murrayvonNeumann2,vonNeumann3}. Note that by Fell's absorption principle we have a conjugacy of unitary representations $\sigma^0 \otimes \lambda \sim 1 \otimes \lambda$, hence it follows that the von Neumann algebra $\{ \sigma_\gamma^0 \otimes \lambda(\gamma) \mid \gamma \in \Gamma \}''$ is canonically isomorphic to $L\Gamma$. 

If the action $\Gamma \actson (B, \eta)$ is measure-preserving, then the crossed product $L^\infty(B, \eta) \rtimes \Gamma$ will be finite and have a canonical normal faithful trace given by $\tau( x ) = \langle x ( \hat{1} \otimes \delta_e ), \hat{1} \otimes \delta_e \rangle$. 

A useful property for finite von Neumann algebras is that for an arbitrary von Neumann subalgebra there always exists a normal conditional expectation \cite{umegaki}. More specifically, if $M$ is a finite von Neumann algebra with normal faithful trace $\tau$, and if $N \subset M$ is a von Neumann subalgebra, then for each $x \in M$ the linear functional $y \mapsto \tau(xy)$ is normal, and hence there is a unique element $E(x) \in L^1(N, \tau)$ such that $\tau(x y) = \tau(E(x) y)$ for all $y \in N$. Since $E(x) \geq 0$ whenever $x \geq 0$ it follows that $E(x) \in N$ for all $x \in M$. This map $E: M \to N$ is a conditional expectation from $M$ to $N$, i.e., it is a unital completely positive projection onto $N$, moreover, it is the unique such map satisfying the condition $\tau \circ E = \tau$ (and this condition also implies that it is normal).

\subsection{On the method of proof} 

Before discussing the outline of our proof we recall Margulis' Normal Subgroup Theorem \cite{margulisamenable, margulisT}: If $\Gamma$ is an irreducible lattice in a center free higher-rank semisimple group $G$ with no compact factors, then $\Gamma$ is just infinite, i.e., every non-trivial normal subgroup has finite index.

Assuming $\Sigma \lhd \Gamma$ is non-trivial, Margulis' strategy for proving that $| \Gamma / \Sigma | < \infty$ consists of two ``halves''. The first, showing that $\Gamma / \Sigma$ has Kazhdan's property (T) \cite{kazhdan}, and the second showing that $\Gamma / \Sigma$ is amenable. Since for countable groups amenability and property (T) together imply finite, the result then follows. In the case when $G$ has property (T), then property (T) for $\Gamma/ \Sigma$ is immediate since property (T) passes to lattices and quotients. Significantly more difficult is the case when $G$ has no factor with property (T), in which case property (T) can be obtained by relating the reduced cohomology spaces of $\Gamma / \Sigma$ to those of $G$ as in \cite{shalomcohom}.

The amenability half of the proof follows by exploiting the amenability properties of the Poisson boundary $G \actson (B, \eta)$ \cite{furstenberg}, together with a ``factor theorem'' showing that any $\Gamma$-quotient of $(B, \eta)$ must actually be a $G$-quotient.\footnote{To avoid confusion with terminology from von Neumann algebras we refer to a $\Gamma$-equivariant map between $\Gamma$-spaces as a $\Gamma$-quotient.} More specifically, the action of $\Gamma$ on $(B, \eta)$ is amenable in the sense of Zimmer \cite{zimmeramenable}, and so there exists a (perhaps non-normal) $\Gamma$-equivariant conditional expectation $E: L^\infty(B, \eta) \overline \otimes \ell^\infty(\Gamma / \Sigma) \to L^\infty(B, \eta)$. If we consider 
$$
L^\infty(B, \eta)^\Sigma = \{ f \in L^\infty(B, \eta) \mid \sigma_{\gamma_0}(f) = f, {\rm \ for \ all \ } \gamma_0 \in \Sigma \},
$$ 
then this is a $\Gamma$-invariant von Neumann subalgebra of $L^\infty(B, \eta)$ and so by Margulis' factor theorem it follows that $L^\infty(B, \eta)^\Sigma$ is also $G$-invariant. However, $\Sigma$ acts trivially on the corresponding Koopman representation and since $G$ is center free, irreducibility for $\Gamma < G$ easily implies that any non-trivial representation of $G$ must be faithful for $\Gamma$. The conclusion is then that $L^\infty(B, \eta)^\Sigma = \mathbb C$, i.e., $\Sigma$ acts ergodically on $(B, \eta)$. However, we have $E( 1 \otimes \ell^\infty(\Gamma / \Sigma) ) \subset L^\infty(B, \eta)^\Sigma = \mathbb C$ and so we conclude that the restriction of $E$ to $1 \otimes \ell^\infty(\Gamma/ \Sigma)$ gives a $\Gamma$-invariant mean, showing that $\Gamma/ \Sigma$ is amenable. 

This general strategy of Margulis is remarkably flexible and has been employed to give rigidity results in a more abstract setting, e.g., \cite{badershalom, creutzshalom}, and also to give classification results beyond normal subgroups, e.g., \cite{stuckzimmer, creutzpeterson, creutz, hartmantamuz} where non-free measure-preserving actions are considered.  


Notions of property (T) and amenability are also of fundamental importance in the theory of II$_1$ factors (see, e.g., \cite{connesclassification, connesjonesT}), and hence it is natural to suspect that Margulis' strategy should have adaptations in this setting as well. This is especially the case given the emergence and success of Popa's deformation/rigidity theory where the major theme is to determine the structure of a II$_1$ factor by contrasting deformation properties (such as amenability) with rigidity properties (such as property (T)), e.g., see \cite{popabetti, popasurvey, vaessurvey, ioanasurvey}. 

Based on the ideas outlined above, if $\Gamma < \Lambda < G$ is as in Theorem~\ref{thm:oascommensurator}, then our strategy to prove operator algebraic superrigidity for $\Lambda$ is to first show that if $\pi: \Lambda \to \mathcal U(M)$ is a representation that generates a finite factor $M$, then either $\pi$ extends to an isomorphism $L\Lambda \to M$ or else the von Neumann subalgebra $N = \pi(\Gamma)''$ is amenable (i.e., injective) and has property (T). Amenability and property (T) for $N$ then imply that $N$ is completely atomic. If $N$ is completely atomic we then show there exists a continuous action of $\rpf{\Lambda}{\Gamma}$ on a finite index von Neumann subalgebra of $M$ that extends conjugation. However, $\rpf{\Lambda}{\Gamma}$ is a product of non-discrete groups with the Howe-Moore property and we rule out the possibility of such actions that are non-trivial.

Just as property (T) passes to quotients of a countable group, it also passes to finite von Neumann algebras that the group generates \cite{connesjonesT}. Thus, even in the finite factor setting if we assume that the ambient group $G$ has property (T), then so does $\pi(\Gamma)''$ and hence the property (T) half of Margulis' strategy follows immediately. In the case when $G$ has a non-compact quotient with property (T), then using the notion of resolutions from \cite{cornulierT}, in a similar fashion as in \cite{creutz}, we show that any representation of $\Gamma$ into an amenable finite von Neumann algebra must, upon passing to a finite index subalgebra, be given by a representation of $G$, which we again rule out using the Howe-Moore property. Thus, in this case also we are reduced to the amenability half of the argument.

The case when $G$ has no non-compact factor with property (T) appears more subtle. (See also the related question at the end of Section 2 in \cite{stuckzimmer}.) While a cohomological characterization of property (T) for finite factors was obtained in \cite{petersonT}, it is less clear if there is a characterization in terms of reduced cohomology. Consequently, it is unclear how to adapt reduced cohomology techniques, e.g., from \cite{shalomcohom}, in the setting of finite factors. This is the reason for our assumption in Theorem~\ref{thm:opalgsuperrigidty} that at least one factor has property (T).

In order to adapt the amenability half of Margulis' strategy to the II$_1$ factor framework, we follow the suggestion of Connes by first understanding the role of the Poisson boundary in this setting \cite{jonesproblems}. A significant first step in this direction was achieved by Izumi \cite{izumi2, izumi4} who introduced the notion of a noncommutative Poisson boundary of a normal unital completely positive map $\phi: \mathcal M \to \mathcal M$ where $\mathcal M$ is a von Neumann algebra. If $\pi: \Gamma \to \mathcal U(M)$ generates $M$, and $\Gamma \actson (B, \eta)$ is the Poisson boundary corresponding to a measure $\mu \in {\rm Prob}(\Gamma)$ then we may consider the ``non-commutative Poisson boundary'' 
\begin{equation}\label{eq:boundary}
\mathcal B_N = \{ \sigma_\gamma^0 \otimes (J \pi(\gamma) J) \mid \gamma \in \Gamma \}' \cap (L^\infty(B, \eta) \overline \otimes \mathcal B(L^2(N, \tau)) ), \tag{\large$\star$}
\end{equation}
which we identify with the von Neumann algebra of all $\Gamma$-equivariant essentially bounded measurable functions $B \to \mathcal{B}(L^{2}(N,\tau))$.

The von Neumann algebra $\mathcal B_N$ is indeed a non-commutative Poisson boundary in the sense of Izumi, although we will not need this fact here. An essential ingredient in our proof though is that by Theorem 5.1 in \cite{zimmerhyperfinite} we have that $\mathcal B_N$ is always an injective von Neumann algebra, and thus we may conclude that $N$ itself is injective if we can somehow conclude that $N = \mathcal B_N$. 

Finishing the analogy with Margulis' normal subgroup theorem we now discuss the analogue of the factor theorem for boundary actions. For this we strengthen the factor theorems for commensurators in \cite{creutzshalom} and \cite{creutzpeterson} to the setting of II$_1$ factors. The key feature of boundaries we use here is that they are contractive (or SAT), i.e., for each measurable set $F \subset B$ we have $\inf_{\gamma \in \Gamma} \eta(\gamma F) \in \{ 0, 1 \}$. This property was introduced by Jaworski \cite{jaworski1, jaworski2}, and we exploit this property to show in Theorem~\ref{thm:uniquemap} that the inclusion of von Neumann algebras $N \subset \mathcal B_N$ is extremely rigid. For example, it follows that the only normal unital $N$-bimodular map from $\mathcal B_N$ to itself is the identity map. This rigidity theorem is a natural extension of Theorem 4.34 in \cite{creutzpeterson} and allows us to show that the ``noncommutative $\Gamma$-quotient'' $\mathcal B_N$ is also invariant under the action of $\Lambda$ (Proposition~\ref{prop:commutantcontainment}). If in addition we have that $\pi$ does not extend to an isomorphism $L\Lambda \to M$, then we use an averaging argument together with the density of $\Lambda$ in $G$ to conclude that $\mathcal B_N$ must in fact be in the commutant of the action of $G$ (Theorem~\ref{thm:smallcommutant}). From this it follows easily that, in fact, $\mathcal B_N = N$ and hence it follows that $N$ is injective.

\section{Operator algebraic rigidity for contractive actions}\label{sec:opalgcontractive}

Throughout this section $\Gamma$ will be a countable group, $\Gamma_0 < \Gamma$ will be a finite index subgroup, $\Gamma \actson (B, \eta)$ will be a contractive action, $N$ will be a finite von Neumann algebra with normal faithful trace $\tau$, and $\pi: \Gamma \to \mathcal U(N)$ will be a homomorphism such that $N = \pi(\Gamma)''$. We let $\mathcal B_N$ be the corresponding von Neumann algebra as defined above by Formula $($\ref{eq:boundary}$)$.

Since $\Gamma \actson (B, \eta)$ is contractive, it is easy to show that for any $f \in L^\infty(B, \eta)$, with $f \geq 0$, there exists a sequence $\gamma_n \in \Gamma$ such that $\sigma_{\gamma_n}(f) \to \| f \|$ in the strong operator topology. Moreover, if $\tilde f \in L^\infty(B, \eta)$ is another function, then we can choose this sequence in such a way so that $\sigma_{\gamma_n}(\tilde f)$ converges strongly to a scalar. The following is an analogue of this fact for the noncommutative situation.

\begin{lem}\label{lem:approxproj}
Using the notation above, for $x \in \mathcal B_N$, $\| x \| \leq 1$, $f \in L^\infty(B, \eta)$, $f \geq 0$, and $\varepsilon > 0$, there exist sequences $\{ g_n \} \subset \Gamma_0$, $\{ p_n \} \subset \mathcal P(N)$, with $\tau(p_n) > 1 - \varepsilon$, for all $n$, and $\{ y_n \} \subset N$, such that $\{ y_n \}$ is uniformly bounded, $\sigma_{g_n}(f)$ converges strongly to  $\| f \|_\infty$, and $\pi(g_n) (p_n x  - y_n) \pi(g_n^{-1})$ converges strongly to $0$.
\end{lem}
\begin{proof}
Let $F_n \subset \Gamma$ be an increasing sequence of finite sets such that $e \in F_1$, and $\cup_n F_n = \Gamma$. For each $n \in \mathbb N$ there exists a measurable subset $E \subset B$ with positive measure such that for $b \in E$ we have $\| f \|_\infty - f(b) < 1/n$, and such that for all $b_1, b_2 \in E$ we have $\| (x(b_1) - x(b_2)) \hat{1} \|_2 \leq 1/n$.

The action of $\Gamma_{0}$ on $(B,\eta)$ is contractive (Proposition 2.3 in \cite{creutzshalom}) and $\{ \gamma \eta : \gamma \in F_{n} \}$ is a finite collection of probability measures all in the same class as $\eta$ so by Lemma 2.5 in \cite{creutzshalom} there exists $g \in \Gamma_{0}$ such that $(\gamma \eta)(gE) > 1 - 1/(n |F_{n}|)$ for all $\gamma \in F_{n}$.  Therefore $\eta(\cap_{\gamma \in F_{n}} \gamma^{-1}gE) > 1 - 1/n$ (note that $g$ depends on $n$).


Fix a point $b_0 \in E$, and set $y_0 = x(b_0) \hat{1} \in L^2(N, \tau)$. Since $\| y_0^* \|_2 = \| y_0 \|_2 \leq \| x \| \leq 1$, if we set $p = \mathds{1}_{[0, \varepsilon^{-1/2}]}(|y_0^*|)$ then by Chebyshev's inequality we have $\tau(p) > 1 - \varepsilon$, and if $y = p y_0$, then $\| y \| \leq \varepsilon^{-1/2}$.

For $h \in F_n$, and $b \in \cap_{\gamma \in F_n} \gamma^{-1} g E$, we have $\| f \|_\infty - \sigma_g(f)(b) < 1/n$, and
\begin{align}
\| \pi(g) ( px(b) - y) \pi(g^{-1})  {\pi(h)} \hat 1 \|_2 
&= \|  (J \pi(g^{-1} h) J ) ( p x(b) - y ) ( J \pi( h^{-1} g ) J ) \hat 1 \|_2 \nonumber \\
&= \| (px(g^{-1} h b) - y ) \hat{1} \|_2 \nonumber \\
&= \| p ( x(g^{-1} h b) - x(b_0) ) \hat{1} \|_2 < 1/n \nonumber
\end{align}

Since the bounds $\| x \| \leq 1$, and $\| y \| \leq \varepsilon^{-1/2}$ are independent of $n$, and as the span of $\pi(\Gamma) \hat{1}$ is dense in $L^2(N, \tau)$, if we set $g_n = g$, $p_n = p$, and $y_n = y$, then as $n \to \infty$ we have that $\sigma_{g_n}(f)$ converges strongly to $\| f \|_\infty$, and ${\pi(g_n) (p_n x - y_n) \pi(g_n^{-1})}$ converges strongly to $0$. 
\end{proof}

The following rigidity theorem for contractive actions strengthens Theorem 4.34 from \cite{creutzpeterson}.

\begin{thm}\label{thm:uniquemap}
Using the notation above, suppose $P \subset \mathcal B_N$ is a von Neumann subalgebra that contains $N$, set
$$
\widetilde{\mathcal B_N} = \{ \sigma_\gamma^0 \otimes (J \pi(\gamma) J ) \mid \gamma \in \Gamma_0 \}' \cap (L^\infty(B, \eta) \overline \otimes \mathcal B(L^2(N, \tau))),
$$ 
so that $N \subset P \subset \mathcal B_N \subset \widetilde{\mathcal B_N}$. If $\Phi: P \to \widetilde{\mathcal B_N}$ is a bounded normal $N$-bimodular unital map, then $\Phi = {\rm id}$.
\end{thm}
\begin{proof}
Fix $x \in P$, and $\varepsilon > 0$. By Lemma~\ref{lem:approxproj} there exist sequences $g_n \in \Gamma_0$, $p_n \in \mathcal P(N)$ with $\tau(p_n) > 1 - \varepsilon$, and $y_n \in N$ such that $\sigma_{g_n}(| \langle (\Phi(x) - x) \hat{1}, \hat{1} \rangle  |) \in L^\infty(B, \eta)$ converges strongly to $\| \langle (\Phi(x) - x) \hat{1}, \hat{1} \rangle \|_\infty$, and $\pi(g_n) (p_nx - y_n) \pi(g_n^{-1})$ converges strongly to $0$.

Since $\Phi$ is normal and $N$-bimodular we then have 
$$
\lim_{n \to \infty} \Phi(\pi(g_n) p_n x \pi(g_n^{-1}) ) - \pi(g_n) p_n x \pi(g_n^{-1}) = \lim_{n \to \infty} \Phi(\pi(g_n) y_n \pi(g_n^{-1}) ) - \pi(g_n) y_n \pi(g_n^{-1}) =  0,
$$ 
where the limit is in the weak operator topology. Thus, using that $\Phi(x) - x \in \widetilde{\mathcal B_N}$ we have
\begin{align}
\| \langle (\Phi(x) - x) \hat{1}, \hat{1} \rangle \|_\infty 
& = \lim_{n \to \infty} \left| \int \sigma_{g_n}( \langle (\Phi(x) - x) \hat{1}, \hat{1} \rangle ) \, d\eta \right| \nonumber \\
& = \lim_{n \to \infty} \left| \int \langle (J \pi(g_n^{-1}) J)(\Phi(x)  - x) (J \pi(g_n) J) \hat{1}, \hat{1} \rangle \, d\eta \right| \nonumber \\
&= \lim_{n \to \infty} \left| \int \langle \pi(g_n)(\Phi(x)  - x) \pi(g_n^{-1}) \hat{1}, \hat{1} \rangle \, d\eta \right| \nonumber \\
&= \lim_{n \to \infty} \left| \int \langle (\Phi(\pi(g_n)x\pi(g_n^{-1}))  - \pi(g_n)x\pi(g_n^{-1})) \hat{1}, \hat{1} \rangle \, d\eta \right| \nonumber \\
&= \lim_{n \to \infty} \left| \int \langle (\Phi( \pi(g_n) (1 - p_n) x \pi(g_n^{-1}) ) - \pi(g_n) (1 - p_n) x \pi(g_n^{-1}) ) \hat{1}, \hat{1} \rangle d\eta \right| \nonumber \\
&\leq \limsup_{n \to \infty} \| \Phi(x) - x \|_\infty \| 1 - p_n \|_2 
< \| \Phi(x) - x \|_\infty \sqrt{\varepsilon}. \nonumber
\end{align}
Since $\varepsilon > 0$ was arbitrary, we conclude that $\langle (\Phi(x) - x) \hat{1}, \hat{1} \rangle$ is $0$ almost everywhere.

If $a, b \in N$, it then follows
$$
\| \langle (\Phi(x) - x) a \hat{1}, b \hat{1} \rangle \|_\infty
= \| \langle (\Phi(b^*xa) - b^*xa ) \hat{1}, \hat{1} \rangle \|_\infty
= 0,
$$
and since $N \subset L^2(N, \tau)$ is dense we then have $\Phi(x) = x$, and so $\Phi = {\rm id}$ since $x$ was arbitrary.
\end{proof}

\begin{cor}\label{cor:finiteindexprojection}
Using the notation above, we have $J(\pi(\Gamma_0)' \cap N)J \subset \mathcal B_N'$. In particular, $\mathcal B_N \subset \mathcal Z(N)' \cap L^\infty(B, \eta) \overline \otimes \mathcal B(L^2(N, \tau))$.
\end{cor}
\begin{proof}
It is enough to show that projections $p \in  J(\pi(\Gamma_0)' \cap N)J \subset L^\infty(B, \eta) \overline \otimes \mathcal B(L^2(N, \tau))$, commute with $\mathcal B_N$. Consider the map $\mathcal B_N \ni x \mapsto \Phi(x) = p x p + (1 - p) x (1 - p)$. Then $\Phi$ is a normal unital completely positive map that restricts to the identity on $N$ and hence is $N$-bimodular. Moreover, $\Phi(x) \in \{ \sigma^0_\gamma \otimes (J \pi(\gamma) J) \mid \gamma \in \Gamma_0 \}'$ since $p \in \mathcal P(J(\pi(\Gamma_0)' \cap N)J)$. Hence, by Theorem~\ref{thm:uniquemap} we have $\Phi = {\rm id}$, i.e., $p \in \mathcal B_N'$.
\end{proof}

\section{Actions of Howe-Moore groups on finite von Neumann algebras}

Recall, that an automorphism $\theta$ of a von Neumann algebra $M$ is properly outer if there is no non-zero element $v \in M$ such that $\theta(x) v = v x$ for all $x \in M$. An action $\alpha: \Lambda \to {\rm Aut}(M, \tau)$ of a countable group $\Lambda$ is properly outer if $\alpha_\lambda$ is properly outer for each $\lambda \in \Lambda \setminus \{ e \}$.

\begin{prop}\label{prop:freerestriction}
Let $G$ be a locally compact group, suppose that $\Lambda < G$ is a countable dense subgroup such that either
\begin{enumerate}[$($i$)$]
\item\label{item:freeA} $G$ is a product of non-compact connected simple groups with the Howe-Moore property, and the $\Lambda$ intersection with any proper subproduct of $G$ is trivial; or
\item\label{item:freeB} $G$ is a simple group with the Howe-Moore property, and $\Lambda$ contains and commensurates a square-integrable lattice $\Gamma < G$, which has a non-torsion element.
\end{enumerate}
Suppose $\alpha: G \to {\rm Aut}(M_0, \tau)$ is a continuous ergodic, trace preserving action of $G$ on a non-trivial finite von Neumann algebra $M_0$ with normal faithful trace $\tau$, then the restriction of $\alpha$ to $\Lambda$ is properly outer.
\end{prop}
\begin{proof}
Suppose $g \in G$, and $v \in M_0$ is such that $\alpha_g(x) v = v x$ for all $x \in M_0$, then it follows that $| v | \in \mathcal Z(M_0)$ and so replacing $v$ with the partial isometry in its polar decomposition we may assume that $v$ is a partial isometry, and that $v^*v \in \mathcal Z(M_0)$. 

We then have $v = v v^*v = v^*v v$, and so $v^*v \geq v v^*$. As $M_0$ is finite we must have $v^*v = vv^*$, and so $\alpha_g(v^*v) v v^*= v (v^* v) v^* = v v^* = v^*v$. Thus, $\alpha_g(v^*v) \geq v^*v$ and as $\alpha$ is trace preserving we then have $\alpha_g(v^*v) = v^*v$. 

Moreover, if $q \in \mathcal Z(M_0)$ such that $q \leq v^*v = vv^*$ then we have $\alpha_g(q) = \alpha_g(q) v v^* = v q v^* = q$. Therefore $\alpha_g$ acts trivially on $\mathcal Z(M_0) v^*v$. Hence, if $\Lambda$ does not act properly outerly, and if $\mathcal Z(M_0) \not= \mathbb C$ then the restriction of the $\Lambda$-action to $\mathcal Z(M_0)$ is not free.

For case (\ref{item:freeA}), Theorem 7.3 in \cite{creutzpeterson} gives that the $\Lambda$-stabilizers are finite almost everywhere and indeed the proof shows that they are contained in the kernel of the $G$-action but here $\Lambda \cap \mathrm{ker}(\alpha) = \{ e \}$, so the action being nonfree is a contradiction.

For case (\ref{item:freeB}), as $H = \{ h \in G : \alpha_{h} \in \mathrm{Inn}(M_{0}) \}$ is a normal subgroup of $G$, if $\alpha_{\lambda}$ is not properly outer for some $\lambda \ne e$ then $H$ is nontrivial hence $H = G$ as $G$ is simple.  Theorem 7.9 in \cite{creutzpeterson} gives that the $\Lambda$-stabilizers are finite almost everywhere.  As there only countably many finite subgroups of $\Lambda$, there must exist some finite $F < \Lambda$ which stabilizes all points in a positive measure set $E_{F}$.  Since all points in $gE_{F}$ are stabilized by $gFg^{-1}$ and $gE_{F}$ has the same measure as $E_{F}$, there are a finite number of finite groups which are stabilizers.  So the set of $\lambda$ for which $\alpha_{\lambda}$ is not properly outer is finite and so we would have $\Lambda = \Lambda \cap H$ is finite, contradicting the existence of a nontorsion element.

Thus, we have left to consider the case when $M_0$ is a factor.
Suppose $h \in G$ and $v \in M_0$ is non-zero partial isometry such that 
\begin{align}\label{eq:inner}
\alpha_h(x) v = v x,
\end{align}
for all $x \in M_0$. Then since $M_0$ is a factor we have $v^*v = 1$, and so $v$ is a unitary. 
If $g \in G$, then replacing $x$ with $\alpha_g(x)$ and applying $\alpha_{g^{-1}}$ to (\ref{eq:inner}) gives 
\begin{align}\label{eq:conj}
\alpha_{g^{-1} h g}(x) \alpha_{g^{-1}}(v) = \alpha_{g^{-1}}(v) x,
\end{align} 
for all $x \in M_0$.

If we assume by contradiction that the action restricted to $\Lambda$ is not properly outer, then we have that $H = \{ h \in G \mid \alpha_h \in {\rm Inn}(M_0) \}$ is a normal subgroup of $G$, which has non-trivial intersection with $\Lambda$, and hence by hypothesis must be dense.

We now focus on cases (\ref{item:freeA}) and (\ref{item:freeB}) separately. In case (\ref{item:freeA}), by projecting down to the quotient $G / \ker(\alpha)$ we may assume that the action is faithful (note that by hypothesis we have that $\Lambda \cap \ker(\alpha) = \{ e \}$). If $G_0 < G$ is a non-trivial simple factor of $G$, then we denote by $H_0$ the subgroup of $H$ consisting of those elements $h \in H$ such that the unitary $v$ in (\ref{eq:inner}) is fixed by $G_0$. It follows easily from (\ref{eq:conj}) that $H_0$ is a normal subgroup of $G$. Note that if we take the intersection over all such subgroups $H_0$ as we vary the factor, then by ergodicity we obtain the trivial group. Thus, we may choose a factor $G_0$, so that $H \not= H_0$.

We may decompose $G$ as $G = \widehat{G_0} \times G_0$, and if $h \in H \setminus H_0$, and $v \in \mathcal U(M_0)$ such that $\alpha_h(x) v = v x$ for all $x \in M_0$, then writing $h$ as $(g_1, g_2) \in \widehat{G_0} \times G_0$ we have 
$$
\alpha_{h}(x) \alpha_{g_2^{-1}}(v) 
= \alpha_{g_2^{-1}}( \alpha_h(\alpha_{g_2}(x)) v) 
= \alpha_{g_2^{-1}}(v \alpha_{g_2}(x))
= \alpha_{g_2^{-1}}(v) x,
$$ 
for all $x \in M_0$. Combining this with (\ref{eq:inner}) then gives $v^* \alpha_{g_2^{-1}}(v) \in \mathcal Z(M_0) = \mathbb C$, and so we have $\alpha_g(v) \in \mathbb T v$ for all $g \in \overline{\langle g_2 \rangle}$. 

Suppose $\overline{\langle g_{2} \rangle}$ is not compact so there exists $g_{n} \in \overline{\langle g_{2} \rangle}$ with $g_{n} \to \infty$ in $G_{0}$.  As $\alpha_{g_{n}}(v) \in \mathbb{T}v$, there is $\{ n_{t} \}$ such that $\alpha_{g_{n_{t}}}(v) \to av$ for some $a \in \mathbb{T}$.  Since $G_{0}$ has the Howe-Moore property and $g_{n_{t}} \to \infty$, this would mean that all of $G_{0}$ fixes $v$ contradicting that $h \notin H_{0}$.

Since the projection of $H \setminus H_0$ onto $G_0$ is dense ($H_0$ is normal and so either the projection of $H_0$ to $G_0$ is trivial, or else the projection is dense and so a non-trivial $H_0$ coset will project densely) it then follows that $G_0$ has a dense set of elements that generate precompact subgroups. This then gives a contradiction since $G_0$ is connected and is a product of groups with the Howe-Moore property hence must be a connected real Lie group \cite{rothman}, and Theorem 3 in \cite{platonov} shows that there then cannot be a dense set of $g \in G_0$ such that $\overline{\langle g_0 \rangle}$ is compact. 

For case (\ref{item:freeB}) we are assuming that $\Lambda$ contains and commensurates a square-integrable lattice $\Gamma$ that has a non-torsion element $\gamma_{0}$ and that $G$ is simple.  As $H$ is normal in $G$, then $H = G$ and in particular $\gamma_{0} \in H$ so there exists $v \in \mathcal U(M_0)$ such that $\alpha_{\gamma}(x) v = v x$ for all $\gamma \in \langle \gamma_0 \rangle$.  

If $v \in \mathbb{T}$ then $\alpha_{\gamma_{0}}(x) = x$ for all $x \in M_{0}$ meaning $\gamma_{0} \in \mathrm{ker}(\alpha)$ but $\gamma_{0} \in \Lambda$ and $\Lambda \cap \mathrm{ker}(\alpha) = \{ e \}$ contradicting that $\gamma_{0}$ is not torsion.

Then $\alpha_{\gamma}(v) = v$ for all $\gamma \in \langle \gamma_0 \rangle$ showing that the action restricted to $\Gamma$ is not mixing (as $v \notin \mathbb{T}$ and $\gamma_{0}$ is not torsion). But since $G$ has the Howe-Moore property, the action must be mixing, and hence must be mixing when restricted to $\Gamma$, giving a contradiction.
\end{proof}

We remark that in the proof in part (\ref{item:freeA}) of the previous theorem we did not actually need that $\Lambda$ was dense. Thus, in this case if $\alpha: G \to {\rm Aut}(M_0, \tau)$ is a continuous, ergodic, trace preserving action on a non-trivial finite von Neumann algebra $M_0$, then $\alpha_g$ is properly outer for any $g \in G$ that is not contained in a proper subproduct of $G$.

Part (\ref{item:freeA}) in the previous theorem also generalizes a result by Segal and von Neumann who showed that a simple real Lie group cannot embed continuously into a finite von Neumann algebra \cite{segalvonneumann}. While we will not use it in the sequel, we show here a general extension of Segal and von Neumann's result to the totally disconnected setting. We thank the referee for the elegant and simple proof.

\begin{thm}\label{thm:noembedding}
Let $G$ be a totally disconnected compactly generated group having no proper open normal subgroup. Then there is no non-trivial continuous homomorphism of $G$ into the unitary group of a finite von Neumann algebra.
\end{thm}
\begin{proof}
Let $\pi: G \to \mathcal U(M)$ be a homomorphism into a tracial von Neumann algebra $(M, \tau)$.  Set $H = \ker(\pi)$, a closed normal subgroup, and let $q: G \to G/H$ denote the quotient map. The sets $E_\varepsilon = \{ g \in G \mid | 1 - \tau(\pi(g)) | < \varepsilon \}$ then give a decreasing family of open sets.  Let $K < G/H$ be an arbitrary compact open subgroup and $C$ a symmetric compact generating set for $G/H$.  Then $C(K) = \cup_{g \in C}~gKg^{-1}$ is relatively compact so there exists $\epsilon > 0$ such that $(C(K) \setminus K) \cap q(E_{\epsilon}) = \emptyset$.  For $g \in C$ we then have $g(q(E_{\epsilon}) \cap K)g^{-1} \subset q(E_{\epsilon}) \cap C(K) = q(E_{\epsilon}) \cap K$.  As $C$ is symmetric, $q(E_{\epsilon}) \cap K$ is then a $C$-invariant set hence $G/H$-invariant.  As $q(E_{\epsilon})\cap K$ is open, $K$ then contains an open normal subgroup of $G/H$.  As $G$ has no proper open normal subgroup, neither does $G/H$ so we conclude that $G/H = K$. Since $K < G/H$ was an arbitrary open subgroup, it then follows that $H = G$, i.e., $\pi$ is the trivial homomorphism.
\end{proof}

If $G$ is a Polish group, $\Lambda$ is a countable group, and we have a homomorphism $\iota: \Lambda \to G$ with dense image, then given a representation $\pi: \Lambda \to \mathcal U(M)$ into a finite von Neumann algebra we may consider 
$$
M_0 = \{ x \in M \mid \| \pi(\lambda_n) x \pi(\lambda_n^{-1}) - x \|_2 \to 0 { \rm \ whenever \ } \iota(\lambda_n) \to e { \rm \ in \ } G \}.
$$ 
Note that $M_0 \subset M$ is a von Neumann subalgebra such that $\pi(\lambda) M_0 \pi(\lambda^{-1}) = M_0$ for all $\lambda \in \Lambda$. We therefore obtain a continuous action $\alpha: G \to {\rm Aut}(M_0, \tau)$ by defining $\alpha_g(x) = \lim_{\iota(\lambda) \to g, \lambda \in \Lambda} \pi(\lambda) x \pi(\lambda^{-1})$. We will call $M_0$ the $G$-algebra (with respect to the map $\iota: \Lambda \to G$) of the representation $\pi$.

\begin{lem}\label{lem:continuouscore}
Let $G$ be a Polish group, suppose that $\Lambda < G$ is a countable dense subgroup such that any trace preserving ergodic action of $G$ on a finite von Neumann algebra is properly outer when restricted to $\Lambda$. Then for any representation $\pi: \Lambda \to \mathcal U(M)$ into a finite factor $M$, such that $\pi(\Lambda)'' = M$, the $G$-algebra of $\pi$ is $\mathbb C$.
\end{lem}
\begin{proof}
Let $M_0 \subset M$ be the $G$-algebra of $\pi$, and let $\alpha: G \to {\rm Aut}(M_0, \tau)$ be the associated continuous action as described above. Note that this is ergodic since $M$ is a factor. We denote by $E_0$ the trace preserving conditional expectation from $M$ to $M_0$. If $M_0 \not= \mathbb C$, then by assumption we have that the action $\alpha$ restricted to $\Lambda$ is properly outer. If $x \in M_0$, and $\lambda \in \Lambda$ then we have $\alpha_\lambda(x) E_0(\pi(\lambda) ) = E_0( \alpha_\lambda(x) \pi(\lambda) ) = E_0( \pi(\lambda) x ) = E_0( \pi(\lambda) ) x$, and hence since the action of $\Lambda$ is properly outer we must have $E_0(\pi(\lambda)) = 0$ for each $\lambda \in \Lambda \setminus \{ e \}$. Since $M = \pi(\Lambda)''$ we would then have $M_0 = E_0(M) = \mathbb C$, giving a contradiction.
\end{proof}

\begin{thm}\label{thm:smallcommutant}
Let $G$ be a Polish group, and $\Lambda < G$ a countable dense subgroup such that each proper closed normal subgroup of $G$ intersects trivially with $\Lambda$. Suppose $G \actson (Y, \eta)$ is ergodic, and $\pi: \Lambda \to \mathcal U(M)$ is a representation into a finite factor such that $M = \pi(\Lambda)''$, and such that the $G$-algebra with respect to $\pi$ is $\mathbb C$. If $N \subset M$ is a von Neumann subalgebra, and $\pi$ does not extend to an isomorphism $L\Lambda \xrightarrow{\sim} M$, then 
$$
\{ \sigma_\lambda^0 \otimes (J E_N(\pi(\lambda)) J) \mid \lambda \in \Lambda \}' \cap L^\infty(Y, \eta) \overline \otimes \mathcal B(L^2N) = 1 \otimes N.
$$
\end{thm}
\begin{proof}
Set $\mathcal Q = \{ \sigma_\lambda^0 \otimes (J E_N(\pi(\lambda)) J) \mid \lambda \in \Lambda \}''$ (which is a subalgebra of $\mathcal{B}(L^{2}(Y,\eta) \otimes L^{2}N)$). As $\pi$ does not extend to an isomorphism $L\Lambda \xrightarrow{\sim} M$ we may fix $\lambda_0 \in \Lambda \setminus \{ e \}$ such that $|\tau(\pi(\lambda_0))| > 0$. For each non-empty open set $O \subset G$ we define $\mathcal K_O = \overline{co} \{ \pi(h \lambda_0 h^{-1}) \mid h \in \Lambda \cap O \}$ where the closure is taken in the $\| \cdot \|_2$-topology (which is equal to the closure in the weak operator topology since $\mathcal K_O$ is bounded in the uniform norm, and convex). We let $\mathcal K = \cap_{O \in \mathcal N(e)} \mathcal K_O$, where $\mathcal N(e)$ is the space of all open neighborhoods of the identity in $G$.

Since $\mathcal K$ is a non-empty $\| \cdot \|_2$-closed convex set it has a unique element $x \in \mathcal K$ that minimizes $\| \cdot \|_2$, and note that $x \not= 0$ since $\tau(y) = \tau(\pi(\lambda_0)) \not= 0$ for each $y \in \mathcal K$. If $\{ \lambda_n \} \subset \Lambda$ is a sequence such that $\lambda_n \to e$ in $G$, then as $\pi(\lambda_n) \mathcal K_{\lambda_n^{-1} O} \pi(\lambda_n^{-1}) = \mathcal K_O$, we have that for each $O \in \mathcal N(e)$, there is large enough $N \in \mathbb N$ such that $\pi(\lambda_n) x \pi(\lambda_n^{-1}) \in \mathcal K_O$ for all $n \geq N$. Consequently, if $y$ is any weak operator topology cluster point of the sequence $\{ \pi(\lambda_n) x \pi(\lambda_n^{-1}) \}$, then $y \in \mathcal K$, and $\| y \|_2 \leq \| x \|_2$, which implies $y = x$ by uniqueness.

Therefore $\pi(\lambda_n) x \pi(\lambda_n)$ converges to $x$ in the weak operator topology and hence 
$$
\| \pi(\lambda_n) x \pi(\lambda_n) - x \|_2^2 = 2 \| x \|_2^2 - 2 \Re(\langle \pi(\lambda_n) x \pi(\lambda_n), x \rangle) \to 0.
$$ 
Hence $x$ is in the $G$-algebra of $\pi$, which is $\mathbb C$ by hypothesis, and so $x = \tau(\pi(\lambda_0)) \in \mathbb C$.

We will now prove that $\sigma_{\lambda_0}^0 \otimes 1 \in \mathcal Q$. Indeed, suppose $\varepsilon > 0$, and we have vectors $\xi_1 \in L^2(Y, \eta)$, and $\xi_2 \in N \subset L^2M$, such that $\| \xi_1 \|_2, \| \xi_2 \|_\infty \leq 1$. Then by continuity of the $G$ action on $(Y, \eta)$ there exists an open neighborhood $O \in \mathcal N(e)$ such that $\| \sigma_{g \lambda_0 g^{-1}}^0( \xi_1) - \sigma_{\lambda_0}^0(\xi_1) \|_2 < \varepsilon$ for all $g \in O$. And from above, there exists a convex combination $\sum_{j = 1}^n \alpha_j \pi(\lambda_j \lambda_0 \lambda_j^{-1})$ such that $\lambda_j \in O$ for all $1 \leq j \leq n$, and $\| \sum_{j = 1}^n \alpha_j \pi(\lambda_j \lambda_0 \lambda_j^{-1}) - \tau(\pi(\lambda_0))  \|_2 < \varepsilon$. Hence,
\begin{align}
&\| ( \sigma^0_{\lambda_0} \otimes \tau(\pi(\lambda_0))  - \sum_{j = 1}^n \alpha_j \sigma^0_{\lambda_j \lambda_0 \lambda_j^{-1}} \otimes E_N(\pi(\lambda_j \lambda_0 \lambda_j^{-1})) ) (\xi_1 \otimes \xi_2) \|_2 \nonumber \\
& \hspace{.2in} \leq \| \sum_{j = 1}^n \alpha_j  \sigma_{\lambda_{0}}^{0} \otimes E_N(\pi(\lambda_j \lambda_0 \lambda_j^{-1}))  - \sum_{j = 1}^n \alpha_j \sigma^0_{\lambda_j \lambda_0 \lambda_j^{-1}} \otimes E_N(\pi(\lambda_j \lambda_0 \lambda_j^{-1})) ) (\xi_1 \otimes \xi_2) \|_2 \nonumber \\
& \hspace{.6in} + \| ( \sigma^0_{\lambda_0} \otimes \tau(\pi(\lambda_0))  - \sum_{j = 1}^n \alpha_j  \sigma_{\lambda_{0}}^{0} \otimes E_N(\pi(\lambda_j \lambda_0 \lambda_j^{-1})) ) (\xi_1 \otimes \xi_2) \|_2 \nonumber \\
& \hspace{.2in} \leq \sum_{j = 1}^n \alpha_j \| (\sigma_{\lambda_0}^0 - \sigma_{\lambda_j \lambda_0 \lambda_j^{-1}}^0)(\xi_1) \|_2 \| \xi_2 \|_2\nonumber \\
& \hspace{.6in} + \| \sigma_{\lambda_0}^0 \otimes (\tau(\pi(\lambda_0)) - \sum_{j = 1}^n \alpha_j E_N(\pi(\lambda_j \lambda_0 \lambda_j^{-1})) ) (\xi_1 \otimes \xi_2) \|_2 \nonumber \\
& \hspace{.2in} <  \varepsilon + \| E_N( \tau(\pi(\lambda_0)) - \sum_{j = 1}^n \alpha_j \pi(\lambda_j \lambda_0 \lambda_j^{-1})  \xi_2 \|_2 < \varepsilon + \varepsilon \| \xi_2 \|_\infty \leq 2 \varepsilon. \nonumber
\end{align}
As the operators above are uniformly bounded, and the span of vectors of the form $\xi_1 \otimes \xi_2$ is dense in $L^2(Y, \eta) \overline \otimes L^2N$,  we then have that $\sigma_{\lambda_0}^0 \otimes \tau(\pi(\lambda_0))$ (and hence also $\sigma_{\lambda_0}^0 \otimes 1$ since $\tau(\pi(\lambda_0)) \not= 0$) is in the strong operator closure of $\{ \sigma_\lambda^0 \otimes (J E_N(\pi(\lambda)) J) \mid \lambda \in \Lambda \}$.

We have therefore shown that $\sigma_g^0 \otimes 1 \in \mathcal Q$ whenever $g \in \Lambda$ such that $\tau(\pi(g) ) \not= 0$. As the set of such $g \in \Lambda$ is preserved under conjugation, we then have that the non-trivial subgroup $\Lambda_0 < \Lambda$ they generate is normal, and since the $\Lambda$ intersection with any proper subproduct of $G$ is trivial we then have $\overline{\Lambda_0} = G$ showing that $\sigma_g^0 \otimes 1 \in \mathcal Q$ for all $g \in G$, and hence we have $1 \otimes (J E_N(\pi(\lambda)) J) \in \mathcal Q$, for all $\lambda \in \Lambda$.

Since $\sigma_g^0 \otimes 1 \in \mathcal Q$ for all $g \in G$, it follows from ergodicity of the $G$ action on $(Y, \eta)$ that $\mathcal Q' \cap L^\infty(Y, \eta) \overline \otimes \mathcal B(L^2N) \subset 1 \otimes \mathcal B(L^2N)$. Also, since $\pi(\Lambda)'' = M$ we have that $\{ E_N(\pi(\lambda)) \mid \lambda \in \Lambda \}$ spans a strong operator topology dense subset of $N$, hence $\mathcal Q' \cap 1 \otimes \mathcal B(L^2N) \subset 1 \otimes (J N J)' = 1 \otimes N$.
\end{proof}

\section{Finite factor representations restricted to the lattice}

\begin{prop}\label{prop:commutantcontainment}
Suppose $G$ is a second countable locally compact group, and $\Gamma < \Lambda < G$ where $\Gamma < G$ is a lattice, and $\Lambda < G$ is a countable dense subgroup that contains and commensurates $\Gamma$. Suppose also that $\pi: \Lambda \to \mathcal U(M)$ is a finite von Neumann algebra representation such that $\pi(\Lambda)'' = M$, and set $N = \pi(\Gamma)''$. Let $G \actson (B, \eta)$ be a quasi-invariant action that is contractive when restricted to $\Gamma$. Then the noncommutative Poisson boundary
\[
\mathcal{B}_{N} = \{ \sigma_\gamma^0 \otimes (J \pi(\gamma) J) \mid \gamma \in \Gamma \}' \cap L^\infty(B, \eta) \overline \otimes \mathcal B(L^2(N, \tau)) \nonumber
\]
has the property that
\begin{align}
\mathcal{B}_{N}
&= \{ \sigma_\lambda^0 \otimes (J E_N(\pi(\lambda)) J) \mid \lambda \in \Lambda \}' \cap L^\infty(B, \eta) \overline \otimes \mathcal B(L^2(N, \tau)). \nonumber
\end{align}
\end{prop}
\begin{proof}
Fix $\lambda \in \Lambda$ and consider the polar decomposition $E_N(\pi(\lambda)) = v_\lambda | E_N(\pi(\lambda)) |$. If we set $\Gamma_0 = \Gamma \cap \lambda \Gamma \lambda^{-1}$, then we have $E_N(\pi(\lambda^{-1})) E_N(\pi(\lambda))  \in \pi(\lambda^{-1}\Gamma_0\lambda)' \cap N$, and hence $| E_N(\pi(\lambda)) | \in \pi(\lambda^{-1}\Gamma_0\lambda)' \cap N$. Thus, for $\gamma \in \Gamma_0$ we have $ \pi(\gamma)v_\lambda =  v_\lambda \pi(\lambda^{-1} \gamma \lambda)$, and taking adjoints then also gives $v_\lambda^* \pi(\gamma^{-1}) = \pi(\lambda^{-1} \gamma^{-1} \lambda ) v_\lambda^*$. If we define $p_\lambda = v_\lambda v_\lambda^* \in \mathcal P(N)$, then we have $p_\lambda \in \pi(\Gamma_0)' \cap N$, and hence by Corollary~\ref{cor:finiteindexprojection}, $J p_\lambda J, J | E_N(\pi(\lambda)) | J \in \mathcal B_N'$. Similarly, if we define $q_\lambda = v_\lambda^* v_\lambda$ then we also have $J q_\lambda J \in \mathcal B_N'$.

Define the map $\Phi : \mathcal B_N \to L^\infty(B, \eta) \overline \otimes \mathcal B(L^2(N, \tau))$, by 
$$
\Phi(x) = \sigma_\lambda \otimes {\rm Ad}(J v_\lambda J) ( x ) + (1 - J p_\lambda J) x.
$$ 
It is easy to see that $\Phi$ is $N$-bimodular normal unital completely positive and if $\gamma \in \Gamma_0$, and $x \in \mathcal B_N$ then we have 
\begin{align}
& \sigma_\gamma \otimes {\rm Ad}(J \pi(\gamma) J) (\Phi(x)) \nonumber \\
& \hspace{.2in} = \sigma_{\gamma \lambda} \otimes {\rm Ad}(J \pi(\gamma) v_\lambda J ) (x) + \sigma_{\gamma} \otimes {\rm Ad} ( J \pi(\gamma) J) ( (1 - J p_\lambda J) x ) \nonumber \\
& \hspace{.2in} = \sigma_\lambda \otimes {\rm Ad}( J v_\lambda J ) ( \sigma_{\lambda^{-1}\gamma \lambda} \otimes {\rm Ad}(J \pi(\lambda^{-1} \gamma \lambda) J ) (x) ) \nonumber \\
& \hspace{.6in} + ( 1 - J p_\lambda J ) \sigma_\gamma \otimes {\rm Ad} ( J \pi(\gamma) J ) (x) \nonumber \\
& \hspace{.2in} 
= \Phi(x). \nonumber
\end{align}
Hence $\Phi: \mathcal B_N \to \widetilde{\mathcal B_N} = \{ \sigma_\gamma^0 \otimes (J E_N(\pi(\gamma)) J) \mid \gamma \in \Gamma_0 \}' \cap L^\infty(B, \eta) \overline \otimes \mathcal B(L^2(N, \tau))$, and by Theorem~\ref{thm:uniquemap} we then have $\Phi = {\rm id}$. 

Hence, for $x \in \mathcal B_N$ we have $\sigma_\lambda^0 \otimes (J v_\lambda J) x \sigma_{\lambda^{-1}}^0 \otimes (J v_{\lambda}^* J) = J p_\lambda J x = x J p_\lambda J$. Multiplying on the right by $\sigma_\lambda^0 \otimes J v_\lambda J$, and we then have $\sigma_\lambda^0 \otimes (J v_\lambda J) x = x \sigma_\lambda^0 \otimes (J v_\lambda J)$. As, $J | E_N(\pi(\lambda)) | J \in \mathcal B_N'$, the result then follows.
\end{proof}

\begin{lem}\label{dlattice}
Suppose $G$ is a locally compact group that is a product of simple groups, and $\Gamma < G$ an irreducible lattice.  Then the intersection of $\Gamma$ with any proper subproduct of $G$ consisting solely of nondiscrete simple groups is trivial.
\end{lem}
\begin{proof}
Write $G = G_{1} \times G_{2}$ and set $N = \Gamma \cap G_{1} \times \{ e \}$.  Since $G_{1} \times \{ e \} \normal G_{1} \times G_{2}$, we have $N \normal \Gamma$.  Let $M < G_{1}$ such that $N = M \times \{ e \}$.  Then $M$ is discrete in $G_{1}$ so
\[
M = \overline{M} = \overline{\mathrm{proj}_{G_{1}}~N} \normal \overline{\mathrm{proj}_{G_{1}}~\Gamma} = G_{1}.
\]
As $G_{1}$ is a product of simple nondiscrete groups, $M$, being discrete in $G_{1}$, must therefore be trivial.
\end{proof}

\begin{lem}\label{dcomm}
Suppose $G$ is a locally compact group that is a product of simple nondiscrete groups, and $\Lambda < G$ a countable dense subgroup that contains and commensurates an irreducible lattice $\Gamma < G$.  Then the intersection of $\Lambda$ with any proper subproduct of $G$ is trivial.
\end{lem}
\begin{proof}
By Theorem 9.2 in \cite{creutzpeterson}, $\Lambda$ sits as an irreducible lattice in $G \times \rpf{\Lambda}{\Gamma}$.  The result then follows from Lemma \ref{dlattice}.
\end{proof}

\begin{thm}\label{thm:amenablesubalgebra}
Suppose $G$ is a locally compact group that is a product of simple groups with the Howe-Moore property, and $\Lambda < G$ is a countable dense subgroup that contains and commensurates an irreducible lattice $\Gamma  < G$, such that either
\begin{enumerate}[$($i$)$]
\item every simple factor of $G$ is connected; or
\item $G$ is simple and totally disconnected and $\Gamma$ is square-integrable and contains a nontorsion element.
\end{enumerate}
Suppose also that $\pi: \Lambda \to \mathcal U(M)$ is a finite factor representation such that $\pi(\Lambda)'' = M$, and set $N = \pi(\Gamma)''$. If $\pi$ does not extend to an isomorphism $L\Lambda \xrightarrow{\sim} M$, then $N$ is injective.
\end{thm}
\begin{proof}
If we take any Poisson boundary $(B, \eta)$ of $G$ corresponding to a spread out probability measure on $G$ \cite{furstenberg}, then $G \actson (B, \eta)$ is amenable (Theorem 5.2 in \cite{zimmeramenbound}) and contractive (Lemma 2.3 in \cite{jaworski2}), and the restriction to $\Gamma$ is again amenable (Theorem 4.3.5 in \cite{zimmerbook}) and contractive (Proposition 2.4 in \cite{creutzshalom}). Since $\Gamma \actson (B, \eta)$ is amenable, Theorem 5.1 in \cite{zimmerhyperfinite} shows that $\mathcal B_N$ is injective. 

By Proposition~\ref{prop:commutantcontainment}, we have that $\mathcal B_N \subset \{ \sigma_\lambda^0 \otimes (J E_N(\pi(\lambda)) J) \mid \lambda \in \Lambda \}'$, and so by combining Lemma~\ref{lem:continuouscore}, with Proposition~\ref{prop:freerestriction} (using Lemma \ref{dcomm} in the connected case) and Theorem~\ref{thm:smallcommutant}, if $\pi$ does not extend to an isomorphism $L\Lambda \xrightarrow{\sim} M$, then $\mathcal B_N \subset 1 \otimes N$, and hence $1 \otimes N =\mathcal B_N$ is then injective. 
\end{proof}

If, in addition, $G$ has property (T) then the conclusion of the previous theorem can be strengthened. In the sequel we will see that using the notion of resolutions from \cite{cornulierT} this is also the case when $G$ has one non-compact factor with property (T) (see also \cite{creutz}).

\begin{cor}\label{cor:subgrouprigidity}
Suppose $\Gamma < \Lambda < G$ is as in the hypotheses of Theorem~\ref{thm:amenablesubalgebra} and suppose, in addition, that $G$ has property (T). Suppose also that $\pi: \Lambda \to \mathcal U(M)$ is a finite factor representation such that $\pi(\Lambda)'' = M$, and set $N = \pi(\Gamma)''$. If $\pi$ does not extend to an isomorphism $L\Lambda \xrightarrow{\sim} M$, then $N$ is completely atomic.
\end{cor}
\begin{proof}
If $\pi$ does not extend to an isomorphism $L\Lambda \xrightarrow{\sim} M$, then by Theorem~\ref{thm:amenablesubalgebra} we have that $N$ is injective. Since $G$ has property (T) so does $\Gamma$ \cite{kazhdan}, and hence it then follows from Theorem C in \cite{robertson} that $\pi(\Gamma) \subset \mathcal U(N)$ is precompact in the strong operator topology. Thus, it follows from the Peter-Weyl theorem that $N$ is isomorphic to a direct sum of matrix algebras.
\end{proof}

\section{Operator algebraic superrigidity for commensurators}\label{sec:mainresults}

\begin{prop}\label{prop:oneTfactor}
Suppose $G$ is a second countable locally compact group, $H \lhd G$ is a closed normal subgroup, and $\Gamma < G$ is a lattice such that the image of $\Gamma$ is dense in $G/H$. Suppose also that $\pi: \Gamma \to \mathcal U(M)$ is a homomorphism into the unitary group of a finite factor $M$ such that $\pi(\Gamma)'' = M$. For any compact neighborhood of the identity $U \subset G/H$ set $\Gamma_U = \{ \gamma \in \Gamma \mid \gamma H \in U \}$.

If $\pi(\Gamma_U) \subset \mathcal U(M)$ is precompact in the strong operator topology for some compact neighborhood of the identity $U \subset G/H$, then the $G/H$-algebra $M_0 \subset M$ (with respect to the quotient map $\Gamma \to G/H$) has finite index in $M$.
\end{prop}
\begin{proof}
For each compact neighborhood of the identity $U \subset G/H$ let $K_U$ be the strong operator topology closure of $\pi(\Gamma_U)$, and set $K$ be the intersection of all $K_U$, and set $N = K''$. By hypothesis $K$ is a compact group, hence by the Peter-Weyl theorem $N = K''$ is completely atomic. If $\gamma \in \Gamma$ then we have $\pi(\gamma) K_U \pi(\gamma^{-1}) = K_{\gamma U \gamma^{-1}}$ and hence it follows that $\pi(\gamma) N \pi(\gamma^{-1}) = N$ for all $\gamma \in \Gamma$.

If $p \in \mathcal P(N)$ is a minimal central projection then we have $\vee_{\gamma \in \Gamma} \pi(\gamma) p \pi(\gamma^{-1})$ is a non-zero projection, which is central since $\pi(\Gamma)'' = M$. Thus $\vee_{\gamma \in \Gamma} \pi(\gamma) p \pi(\gamma^{-1}) = 1$, and since $\pi(\gamma_1) p \pi(\gamma_1^{-1})$ and $\pi(\gamma_2) p \pi(\gamma_2^{-1})$ are either equal or orthogonal for all $\gamma_1, \gamma_2 \in \Gamma$, it then follows that $\mathcal Z(N)$ is finite dimensional and hence so is $N$.

We set $M_0 = N' \cap M$, which is then a finite index von Neumann subalgebra of $M$. If $\{ \gamma_n \} \subset \Gamma$ is a sequence such that $\gamma_n \to e$ in $G/ H$, then by hypothesis we have that $\{ \pi( \gamma_n) \}$ is precompact in the strong operator topology and hence for any subsequence $\{ \pi( \gamma_{n_k}) \}$ of $\{ \pi( \gamma_n) \}$ there exists a unitary $u \in \mathcal U(N)$ that is a strong operator topology cluster point, and hence for $x \in M_0$ we have that $x = u x u^*$ is a strong operator topology cluster point of $\{ \pi(\gamma_{n_k}) x \pi(\gamma_{n_k}^{-1}) \}$. As the subsequence $\{ \pi(\gamma_{n_k} ) \}$ was arbitrary it then follows that $\pi(\gamma_n)x\pi(\gamma_n^{-1}) \to x$ in the strong operator topology. Thus, we have shown that the finite index subalgebra $M_0 \subset M$ is contained in the $G/H$-algebra. 
\end{proof}

\begin{prop}\label{prop:oneT}
Suppose $G$ and $H$ are locally compact second countable groups that are products of simple non-compact non-discrete groups with the Howe-Moore property, suppose also that $\Gamma < G \times H$ is an irreducible lattice,
such that either
\begin{enumerate}[$($i$)$]
\item every simple factor of $G$ is connected; or
\item $G$ is simple and totally disconnected and $\Gamma$ is square-integrable and contains a nontorsion element.
\end{enumerate}
and $\pi: \Gamma \to \mathcal U(M)$ is a representation into a finite von Neumann algebra with $\pi(\Gamma)'' = M$ and such that either 
\begin{enumerate}
\item $G$ has property (T) and $M$ has the Haagerup property, or
\item $\pi( \Gamma \cap (G \times U) )$ is precompact for all compact neighborhoods of the identity $U \subset H$,
\end{enumerate}
then $M$ is completely atomic.
\end{prop}
\begin{proof}
By considering the integral decomposition of $M$ into factors it is enough to treat the case when $M$ is a factor \cite{thoma1}. For the first case, since $G$ has property (T) it follows from Theorem 1.8 and Proposition 1.11 in \cite{cornulierT} that the subset $\Gamma_U = \Gamma \cap (G \times U) \subset \Gamma$ has relative property (T) for some (and hence all) compact neighborhood of the identity $U \subset H$. The same argument in \cite{connesjonesT} for the case of property (T) groups then implies that $\pi(\Gamma_U)$ is precompact in the strong operator topology (see also \cite{robertson}), and hence we have reduced the problem to the second case. 

For the second case we may then apply Proposition~\ref{prop:oneTfactor} to conclude that the $G$-algebra $M_0 \subset M$ is finite index. However, by Proposition~\ref{prop:freerestriction} and Lemma~\ref{lem:continuouscore} we must have $M_0 = \mathbb C$, and hence $M$ is finite dimensional.
\end{proof}

\begin{thm}\label{thm:oasuperrigidity1}
Suppose $G$ is a locally compact second countable group that is a product of simple non-discrete non-compact groups with the Howe-Moore property, and at least one factor having property (T), suppose also that $\Lambda < G$ is a countable dense subgroup that contains and commensurates an irreducible lattice $\Gamma  < G$, such that either
\begin{enumerate}[$($i$)$]
\item every simple factor of $G$ is connected; or
\item $G$ is simple and totally disconnected and $\Gamma$ is square-integrable and contains a nontorsion element.
\end{enumerate}
and such that $\rpf{\Lambda}{\Gamma}$ is a product of simple groups with the Howe-Moore property.

If $\pi: \Lambda \to \mathcal U(M)$ is a finite factor representation such that $\pi(\Lambda)'' = M$, and if $\pi$ does not extend to an isomorphism $L\Lambda \xrightarrow{\sim} M$, then $M$ is finite dimensional.
\end{thm}
\begin{proof}
If we set $N = \pi(\Gamma)''$ then by Theorem~\ref{thm:amenablesubalgebra} we have that $N$ is injective and hence also has the Haagerup property. By Proposition~\ref{prop:oneT} (Part 1) we then have that $N$ is completely atomic. 

If we consider, as in Theorem 9.2 in \cite{creutzpeterson}, the diagonal lattice embedding $\Lambda \to G \times (\rpf{\Lambda}{\Gamma})$, then again applying Proposition~\ref{prop:oneT} (this time Part 2) gives the result.
\end{proof}

\begin{thm}\label{thm:oasuperrigidity}
Suppose $G$ is a locally compact second countable group that is a product of simple non-discrete non-compact groups with the Howe-Moore property, and at least one factor having property (T), and if there are connected factors then at least one connected factor having property $(T)$.

Suppose also that $\Lambda < G$ is a countable dense subgroup that contains and commensurates an irreducible lattice $\Gamma  < G$ such that $\rpf{\Lambda}{\Gamma}$ is a product of simple non-compact non-discrete groups with the Howe-Moore property.  If $G$ is totally disconnected, assume there is a simple factor $G_{0}$ with property $(T)$ such that for some compact open subgroup $K$ of the product of the other factors ($K$ is trivial if $G$ is simple), the projection of $\Gamma \cap G_{0} \times K$ to $G$ is square-integrable and contains a nontorsion element.

If $\pi: \Lambda \to \mathcal U(M)$ is a finite factor representation such that $\pi(\Lambda)'' = M$, and if $\pi$ does not extend to an isomorphism $L\Lambda \xrightarrow{\sim} M$, then $M$ is finite dimensional.
\end{thm}
\begin{proof}
Write $G = G_{c} \times G_{d}$ where $G_{c}$ are the connected simple groups in $G$ and $G_{d}$ the totally disconnected groups.

First consider when $G_{c}$ is nontrivial.  Let $K$ be a compact open subgroup of $G_{d}$ (take $K = \{ e \}$ if $G_{d}$ is trivial) and let $\Gamma_{c} = \mathrm{proj}_{G_{c}}~(\Gamma \cap G_{c} \times K)$ which is an irreducible lattice in $G_{c}$ (Theorem 9.2 in \cite{creutzpeterson}).  Let $\Lambda_{c} = \mathrm{proj}_{G_{c}}~\Lambda$ so that $\Lambda_{c}$ contains and commensurates $\Gamma_{c}$.  By Lemma \ref{dcomm}, $\Lambda \cap \mathrm{ker}(\mathrm{proj}_{G_{c}}) = \{ e \}$ so $\Lambda_{c}$ is isomorphic to $\Lambda$.  Note that $\rpf{\Lambda_{c}}{\Gamma_{c}}$ is isomorphic to $G_{d} \times \rpf{\Lambda}{\Gamma}$ and is therefore a product of simple non-compact non-discrete simple groups with the Howe-Moore property.  Theorem \ref{thm:oasuperrigidity1} then gives the result.

Now consider when $G_{c}$ is trivial.  Decompose $G = G_{0} \times G_{1}$. Let $\Gamma_{0} = \mathrm{proj}_{G_{0}}~\Gamma$ which is an irreducible lattice in $G_{0}$ that is contained and commensurated by $\Lambda_{0} = \mathrm{proj}_{G_{0}}~\Lambda$ (which is isomorphic to $\Lambda$ by Lemma \ref{dcomm}).  Then $\Lambda_{0}$ is dense in a simple non-compact non-discrete Howe-Moore group with property $(T)$ that contains and commensurates a square-integrable lattice which contains a nontorsion element.  As $\rpf{\Lambda_{0}}{\Gamma_{0}}$ is isomorphic to $G_{1} \times \rpf{\Lambda}{\Gamma}$, Theorem \ref{thm:oasuperrigidity1} gives the result.
\end{proof}

The previous theorem easily implies Theorem~\ref{thm:oascommensurator} from the introduction.

\begin{cor}\label{cor:commensuratorfree}
Suppose $\Gamma < \Lambda < G$ is as in the hypotheses of Theorem~\ref{thm:oasuperrigidity}.
Then any probability measure-preserving ergodic action of $\Lambda$ on a standard Lebesgue space is essentially free.
\end{cor}
\begin{proof}
Since every cyclic representation generating a finite factor must either be conjugate to the left-regular representation, or else finite dimensional, and since, up to conjugacy, there are only countably many finite dimensional representations (Theorem 10.3 in \cite{shalomcohom}), this follows directly from Theorem 2.11 in \cite{dudkomedynetsthompson} or Theorem 3.2 in \cite{petersonthom}.
\end{proof}

\begin{cor}
Let $G$ be a locally compact second countable group that is a product of at least two simple groups with the Howe-Moore property. Suppose that at least one factor of $G$ has property (T), at least one factor is totally disconnected, and if there exists connected factors then at least one should have property (T).

Let $\Gamma < G$ be an irreducible lattice. and, in the case when $G$ is totally disconnected, assume there is a simple factor $G_{1}$ with property $(T)$ such that for some compact open subgroup $K$ of the product of the other factors, the projection of $\Gamma \cap G_{1} \times K$ to $G_{1}$ is square-integrable and contains a nontorsion element.

If $\pi: \Gamma \to \mathcal U(M)$ is a finite factor representation such that $\pi(\Gamma)'' = M$, and if $\pi$ does not extend to an isomorphism $L\Lambda \xrightarrow{\sim} M$, then $M$ is finite dimensional. Moreover, $\Gamma$ has at most countably many finite dimensional irreducible representations, and any probability measure-preserving ergodic action of $\Gamma$ on a standard Lebesgue space is essentially free.
\end{cor}
\begin{proof}
If we write $G = G_1 \times G_2$ where $G_i$ are non-trivial, with $G_2$ being totally disconnected, then if we set $\Gamma_0 = \Gamma \cap ( G_1 \times K )$, then $\Gamma_0$ projects down to a square-integrable lattice in $G_1$, which is commensurated by the projection of $\Gamma$ and is not a torsion group. As in Section 10 of \cite{creutzpeterson}, the result then follows from Theorem~\ref{thm:oasuperrigidity} and Corollary~\ref{cor:commensuratorfree} by considering the inclusion $\Gamma_0 < \Gamma < G_1$.
\end{proof}

The previous corollaries easily imply Theorem~\ref{thm:opalgsuperrigidty} and Corollary~\ref{cora:nonfree} from the introduction. The other results in the introduction follow easily from these results.  

\section*{Acknowledgments} Many of the ideas in this work were developed while J.P.\ was visiting UC, San Diego in the fall of 2012, he is grateful for their hospitality, and he is especially grateful to Adrian Ioana for the many useful conversations regarding this work. J.P.\ is supported by NSF Grant DMS-1201565, and a grant from the Alfred P. Sloan Foundation. 

The authors would like to thank the referees for many helpful comments and for pointing out a simpler, more elegant proof of Theorem \ref{thm:noembedding}.

\small
\bibliographystyle{amsalpha}
\bibliography{ref}

\noindent
\textsc{Department of Mathematics, US Naval Academy, 572C Holloway Road, Annapolis, MD 21402, U.S.A.}

\noindent
\textsc{Department of Mathematics, Vanderbilt University, 1326 Stevenson Center, Nashville, TN 37240, U.S.A.}

\end{document}